\documentclass{article}
\usepackage[utf8]{inputenc}
\usepackage[top=30mm, left=30mm, right=30mm, bottom=30mm]{geometry}
\usepackage{amsmath,amssymb,amsthm}
\usepackage{dsfont}
\usepackage{color}
\usepackage{lineno,hyperref}

\modulolinenumbers[5]

\newtheorem{theorem}{Theorem}[section]

\newtheorem{prop}[theorem]{Proposition}

\theoremstyle{definition}
\newtheorem{dfn}[theorem]{Definition}
\newtheorem{rem}[theorem]{Remark}

\begin{document}
\title{Existence of flows for linear Fokker-Planck-Kolmogorov equations and its connection to well-posedness}

\author{Marco Rehmeier\footnote{Faculty of Mathematics, Bielefeld University, 33615 Bielefeld, Germany. E-Mail: mrehmeier@math.uni-bielefeld.de }}

\date{}
\maketitle
\begin{abstract}
	\noindent Let the coefficients $a_{ij}$ and $b_i$, $i,j \leq d$, of the linear Fokker-Planck-Kolmogorov equation (FPK-eq.)
	$$\partial_t\mu_t = \partial_i\partial_j(a_{ij}\mu_t)-\partial_i(b_i\mu_t)$$ be Borel measurable, bounded and continuous in space. Assume that for every $s \in [0,T]$ and every Borel probability measure $\nu$ on $\mathbb{R}^d$ there is at least one solution $\mu = (\mu_t)_{t \in [s,T]}$ to the FPK-eq. such that $\mu_s = \nu$ and $t \mapsto \mu_t$ is continuous w.r.t. the topology of weak convergence of measures. We prove that in this situation, one can always select one solution $\mu^{s,\nu}$ for each pair $(s,\nu)$ such that this family of solutions fulfills 
	$$\mu^{s,\nu}_t = \mu^{r,\mu^{s,\nu}_r}_t \text{ for all }0 \leq s \leq r \leq t \leq T,$$which one interprets as a \textit{flow property} of this solution family. Moreover, we prove that such a flow of solutions is unqiue if and only if the FPK-eq. is well-posed.
\end{abstract}

\noindent	\textbf{Keywords:} Fokker-Planck-Kolmogorov equation, flow property, martingale problem, superposition principle\\ \\
	\textbf{2010 MSC}: 60J60; 35Q84
	
\section{Introduction}
In this paper we are concerned with linear Fokker-Planck-Kolmogorov equations of the form
\begin{equation}\label{FPK-eq.}
\partial_t\mu_t = \partial_i\partial_j(a_{ij}\mu_t)-\partial_i(b_i\mu_t),
\end{equation}which are second order parabolic equations for measures. Here $a_{ij},b_i : [0,T]\times \mathbb{R}^d \to \mathbb{R}$ are given coefficients with suitable measurability- and regularity-assumptions imposed below. For $s \in [0,T]$ and a Borel probability measure $\nu$ on $\mathbb{R}^d$, we consider the Cauchy problem of (\ref{FPK-eq.}) with initial condition $\mu_s = \nu$. Our notion of a solution to such a Cauchy problem is that of narrowly continuous probability curves, i.e. a family of Borel probability measures $(\mu_t)_{t \in [s,T]}$ such that $[s,T]\ni t \mapsto \mu_t$ is weakly continuous and 
\begin{equation*}
\int_{\mathbb{R}^d}fd\mu_t - \int_{\mathbb{R}^d}fd\nu = \int_{s}^{t}\int_{\mathbb{R}^d}L_uf(x)d\mu_u(x)du
\end{equation*} holds for all smooth, compactly supported $f: \mathbb{R}^d \to \mathbb{R}$. A shorthand notation for equation (\ref{FPK-eq.}) is $$\begin{cases}
\partial_t\mu_t&= L^*\mu_t \\
\mu_s &= \nu
\end{cases},$$ where $L_t(x) := L_{A,b}(t,x):= \frac{1}{2}a_{ij}(t,x)\partial_i\partial_j+b_i(t,x)\partial_i$ denotes the second order differential generator associated to $A:= (a_{ij})_{i,j \leq d}$ and $b:= (b_i)_{i \leq d}$, also called Kolmogorov operator.\\
\\Fokker-Planck-Kolmogorov equations have been an active research topic in the past decades. There is a vast literature on general results such as existence, uniqueness and regularity of solutions, also for a more general notion of equations as (\ref{FPK-eq.}). A thorough analytical introduction into the field is provided by the work \cite{bogachev2015fokker} of Röckner, Krylov, Bogachev and Shaposhnikov from 2015.\\ \\Fokker-Planck-Kolmogorov equations also have strong and fruitful connections to probability theory, in particular to the theory of diffusion processes. For example, the transition probabilities of a typical diffusion process in $\mathbb{R}^d$ with drift $b=(b_i)_{i \leq d}$ and diffusion coefficients $(a_{ij})_{i,j \leq d}$ solve the corresponding Fokker-Planck-Kolmogorov equation, i.e. equation (\ref{FPK-eq.}). \\Above that, equation (\ref{FPK-eq.}) is closely related to the martingale problem associated to the coefficients $a_{ij}$ and $b_i$. More precisely, every continuous solution to the martingale problem provides a narrowly continuous probability solution to  equation (\ref{FPK-eq.}) via its one-dimensional marginals. Conversely, by a so-called \textit{superposition principle} of Trevisan from 2016, which is an extension of an earlier work by Figalli from 2008, c.f. \cite{trevisan2016} and \cite{FIGALLI2008109}, respectively, given a narrowly continuous probability solution $(\mu_t)_{t \in [s,T]}$ to equation (\ref{FPK-eq.}), there exists a continuous solution to the corresponding martingale problem, for which the one-dimensional marginals are given by $(\mu_t)_{t \in [s,T]}$. A fundamental investigation of martingale problems and its connection to Fokker-Planck-Kolmogorov equations can be found in \cite{stroock2007multidimensional} by Stroock and Varadhan from 2006.\\
\\In particular, in this work Stroock and Varadhan prove the following: Given continuous and bounded coefficients $a_{ij}$ and $b_i$, assume there exists at least one continuous solution to the martingale problem with start in $x \in \mathbb{R}^d$ at time $s \geq 0$ for every pair $(s,x)$. Then there exists a strong Markovian selection of such solutions. More precisely, one can select a solution to the martingale problem $P^{s,x}$ for every initial condition $(s,x)$ such that the family $(P^{s,x})_{(s,x) \in \mathbb{R}_+ \times \mathbb{R}^d}$ is a strong Markov process on the space of continuous functions $C(\mathbb{R}_+,\mathbb{R}^d)$. Such a consideration goes back to an earlier work of Krylov from 1973 (c.f. \cite{Krylov_1973}). It is also proven that such a selection is unique if and only if the martingale problem is well-posed.\\
\\The aim of this paper is to prove similar results for the Fokker-Planck-Kolmogorov equation (\ref{FPK-eq.}). Our first main result (see Theorem \ref{Main Thm}) is the following: Assume all coefficients $a_{ij}$ and $b_i$ are Borel measurable, globally bounded and continuous in the spatial variable. Assuming that the Cauchy problem for equation (\ref{FPK-eq.})  has at least one narrowly continuous probability solution for every initial condition $(s,\nu)$, we prove the existence of a family of solutions $(\mu^{s,\nu})_{s, \nu}$ such that \begin{equation}\label{Flow}
\mu^{s,\nu}_t = \mu^{r,\mu^{s,\nu}_r}_t \text{ for all }0 \leq s \leq r \leq t \leq T
\end{equation} and all probability measures $\nu$. We regard (\ref{Flow}) as a flow property for solutions to (\ref{FPK-eq.}). Moreover, in Theorem \ref{second main thm} we show: There exists exactly one such flow if and only if the Fokker-Planck-Kolmogorov equation is well-posed among narrowly continuous probability solutions.\\
\\The structure of this paper is as follows: In Section 2 we introduce notation,  present the exact notion of solution to the Cauchy problem of equation (\ref{FPK-eq.}) and state the assumptions on the coefficients $a_{ij}$ and $b_i$. In the third section, we present our two main results, which are Theorem \ref{Main Thm} and \ref{second main thm}. We set up all necessary notions and tools for its proofs. In particular, this includes the aforementioned superposition principle by Figalli and Trevisan. Afterwards we prove both main theorems.\\
\\It would be interesting to generalize our results to Fokker-Planck-Kolmogorov equations on infinite dimensional state spaces, e.g. replacing $\mathbb{R}^d$ by a separable Hilbert space $H$. The techniques we developed within the proof of Theorem \ref{Main Thm} seem promising for this more general case as well. In order to widen the spectrum of possible applications, it is also desirable to establish our main results under more general assumptions on the coefficients $a_{ij}$ and $b_i$. This could be a direction of further research on this topic.

\section{Preliminaries}
\subsection{Notation}
Let us introduce basic notation, which we will frequently use in the sequel.
\\For a metric space $X$, $\mathcal{P}(X)$ denotes the set of all Borel probability measures on $X$. If $X= \mathbb{R}^d$, we will simply write $\mathcal{P}:= \mathcal{P}(\mathbb{R}^d)$.\\ $C_b(\mathbb{R}^d)$ is the set of all bounded and continuous functions $f: \mathbb{R}^d\to \mathbb{R}$ and $C^{\infty}_{0}(\mathbb{R}^d)$ the set of all such $f$, which are smooth and have compact support. For open sets $I \subseteq \mathbb{R}$ and $ \Omega \subseteq \mathbb{R}^d$, $C_b^{1,2}(I\times\Omega)$ [$C_c^{1,2}(I\times\Omega)$] denotes all bounded [compactly supported] functions $f: I \times \Omega \to \mathbb{R}$, which have at least one bounded continuous derivative w.r.t. $t\in I$ and at least two bounded continuous derviatives w.r.t. $x \in \Omega$. The spatial derivative in the $i$-th euclidean direction for such a function is denoted by $\partial_if$, $i \leq d$, and the derivative w.r.t. the time-variable by $\partial_tf$. As usual, for two topological spaces $X$ and $Y$, $C(X,Y)$ denotes the set of all continuous functions $f: X \to Y$.\\ For a  time interval $I \subseteq \mathbb{R}_+$, a $\textit{Borel curve}$ of Borel (probability) measures on $\mathbb{R}^d$ is a family $(\mu_t)_{t \in I}$ such that $t \mapsto \mu_t(A)$ is Borel measurable for every $A \in \mathcal{B}(\mathbb{R}^d)$. We call such a Borel curve \textit{narrowly continuous}, if $t \mapsto \int fd\mu_t$ is continuous for all $f \in C_b(\mathbb{R}^d)$, i.e. in other words, if the map $t \mapsto \mu_t$ is continuous w.r.t. the topology of weak convergence of measures on $\mathcal{P}$. \\The one-dimensional Lebesgue measure on $\mathbb{R}$ or on an interval $I \subseteq \mathbb{R}$ is denoted by $dt$. The set of all symmetric, non-negative definit $d\times d$-matrices with real entries is denoted by $S_+(\mathbb{R}^d)$.
\subsection{Basic Setting}
Let  $T >0$, which we regard as fixed throughout this paper, and let $$a_{ij}, b_i:[0,T] \times \mathbb{R}^d \to \mathbb{R} \text{ for } i,j \in \{1,...,d\}$$ fulfill the following assumptions: $a_{ij}$ and $b_i$ are Borel measurable, continuous in $x \in \mathbb{R}^d$ and globally bounded such that $A(t,x) :=(a_{ij}(t,x))_{i,j \leq d} \in S_+(\mathbb{R}^d)$ for every $(t,x)\in [0,T]\times \mathbb{R}^{d}$. These assumptions will be in force throughout the entire paper.\\ \\Further let $L:=L_{A,b}$ be the operator defined through $L_tf(t,x) := \frac{1}{2}a_{ij}(t,x)\partial_i\partial_jf(t,x)+b_i(t,x)\partial_if(t,x)$ for every $f: I \times \mathbb{R}^d \to \mathbb{R}$ with at least two spatial derivatives. We always assume summation over repeated indices. 

\begin{dfn}\label{Def sol FPK}
	A \textit{narrowly continuous probability solution to the Cauchy problem for the Fokker-Planck-Kolmogorov equation} (\textit{FPK-eq.}) \textit{w.r.t.} 	$a_{ij},b_i$ \textit{on} $[0,T]\times \mathbb{R}^d$ \textit{with initial condition} $(s,\nu) \in [0,T]\times \mathcal{P}$, i.e. to
	$$\begin{cases} 
	\partial_t\mu_t &= \partial_i\partial_j(a_{ij}\mu_t)-\partial_i(b_i\mu_t)	\\
	\mu_s &= \nu
	\end{cases},$$ sometimes shortly written as
	$$\begin{cases} 
	\partial_t\mu_t &= L^*\mu_t	\\
	\mu_s &= \nu
	\end{cases},$$
	is a narrowly continuous Borel curve of probability measures $\mu = (\mu_t)_{t \in [s,T]} \in C\big([s,T],\mathcal{P}\big)$ such that 
	\begin{equation*}
	\int_{s}^{T}\int_{\mathbb{R}^d}\partial_tf(t,x)+L_tf(t,x)d\mu_t(x)dt = 0
	\end{equation*} for all $f \in C^{1,2}_c\big((s,T)\times \mathbb{R}^d\big)$. Equivalently, we may require
	\begin{equation}
	\int_{\mathbb{R}^d}fd\mu_t - \int_{\mathbb{R}^d}fd\nu = \int_{s}^{t}\int_{\mathbb{R}^d}L_uf(x)d\mu_u(x)du
	\end{equation}for all $f \in C^{\infty}_0(\mathbb{R}^d)$ and all $t \in [s,T]$ . In particular $\mu_s = \nu$.
\end{dfn}
\begin{dfn}
	The set of all narrowly continuous probability solutions to the above FPK-eq. with initial condition $(s,\nu)$ is denoted by $FP(L,s,\nu)$. Since we fix $L= L_{A,b}$ throughout, no  confusion will occur when we write $FP(s,\nu)$ instead of $FP(L,s,\nu)$.
\end{dfn}

\section{The Main Results}
\begin{dfn}\label{Def flow prop}
	Let $\gamma^{s,\nu} = (\gamma^{s,\nu}_t)_{t \in [s,T]} \in FP(s,\nu)$ for every $(s,\nu) \in [0,T] \times \mathcal{P}$. We say that the family $(\gamma^{s,\nu})_{(s,\nu) \in [0,T]\times \mathcal{P}}$ has the \textit{flow property}, if for every $0 \leq s \leq r \leq t \leq T$ and $\nu \in \mathcal{P}$ we have \begin{equation}\label{flow-prop}
	\gamma^{s,\nu}_t = \gamma^{r,\gamma^{s,\nu}_r}_t.
	\end{equation}
\end{dfn}
Our first main result is the following
\begin{theorem}\label{Main Thm}
	Assume $a_{ij}$ and $b_i$ fulfill the measurability- and regularity-conditions imposed above Definition \ref{Def sol FPK}. Further assume that $FP(s,\nu)$ is non-empty for every $(s,\nu)$, i.e. for every initial condition $(s,\nu) \in [0,T]\times \mathcal{P}$ there exists at least one narrowly continuous probability solution to the above FPK-equation with start in $(s,\nu)$. Then there exists a family $(\mu^{s,\nu})_{(s,\nu)\in [0,T]\times \mathcal{P}}$, which has the flow property.
\end{theorem}
For the proof, we need several preparations. \\ \\For $0\leq s \leq T < +\infty$, we set $\mathbb{Q}_s^T:= [s,T]\cap \mathbb{Q}$.
Let $\mathcal{P}^{\mathbb{Q}^T_s} := \big\{(\gamma_q)_{q \in \mathbb{Q}^T_s}|\gamma_q \in \mathcal{P}\big\}$ be endowed with the product topology of weak convergence of probability measures. Note that this space is Polish, since the topology of weak convergence on $\mathcal{P}$ is metrizable by the Prohorov metric.\\
Moreover, the space $C([s,T],\mathbb{R}^d)$ will always be endowed with the norm of uniform convergence.

\begin{dfn}\label{Def measure det}
	Let $I$ be some index set. A family of measurable functions $\{f_i\}_{i \in I}$, $f_i : \mathbb{R}^d\to \mathbb{R}$, is called \textit{measure-determining (for finite Borel measures) on} $\mathbb{R}^d$, if for any two finite Borel measures $\mu_1,\mu_2$ on $\mathbb{R}^d$, $\int f_id\mu_1 = \int f_i d\mu_2$ for all $i \in I$ implies $\mu_1 = \mu_2$. 
\end{dfn}
\begin{rem}
	Clearly every dense, countable subset $\{f_n\}_{n \in \mathbb{N}}$ of $C^{\infty}_0(\mathbb{R}^d)$ is such a measure-determining family.
\end{rem}
\begin{dfn}\label{Def enumeration}
	\begin{enumerate}
		\item [(i)] 	Any bijection $\eta: \mathbb{N}\times \mathbb{Q}^T_0 \to \mathbb{N}_0$ will be called an \textit{enumeration}. Given an enumeration $\eta$, the $k$-th element in $\mathbb{N}\times \mathbb{Q}^T_0$ according to $\eta$ is denoted by $(n_k,q_k)$.
		\item[(ii)]  For $s \in [0,T]$, we denote by $(m^s_k)_{k \in \mathbb{N}_0}$ the enumerating sequence of $\mathbb{N}\times \mathbb{Q}^T_s$ according to $\eta$, i.e. $\eta^{-1}(m^s_k)$ is the $k$-th element in $\mathbb{N}\times \mathbb{Q}^T_s$ according to $\eta$.
	\end{enumerate}
\end{dfn}
\begin{rem}\label{enumeration subsequence}
	Note that for $0 \leq s \leq r \leq T$, the sequence $(m^r_l)_{l \in \mathbb{N}_0}$ is a subsequence of $(m^s_l)_{l \in \mathbb{N}_0}$. This will be important when we verify the flow-property in the proof of Theorem \ref{Main Thm}.
\end{rem}
\begin{dfn}\label{Def J}
	For $s \in [0,T]$ let $J_s: C\big([s,T], \mathcal{P}\big) \to \mathcal{P}^{\mathbb{Q}^T_s}$ be defined by
	\begin{equation*}
	J_s\big((\gamma_t)_{t \in [s,T]}\big) := (\gamma_q)_{q \in \mathbb{Q}^T_s},
	\end{equation*}i.e. it is simply the projection on all coordinates $q \in \mathbb{Q}^T_s$.
\end{dfn}
\begin{rem}\label{Cont J}
	In the sequel we will always consider $C\big([s,T],\mathcal{P}\big)$ with the product topology of weak convergence of probability measures and not, as common, with the topoloy of uniform convergence. Then $J_s$ as in Definition \ref{Def J} is clearly continuous for every $s \in [0,T]$. However, note that $C\big([s,T],\mathcal{P}\big)$ endowed with the product topology of weak convergence of measures is not metrizable.
\end{rem}The next remark points out an important technique of the proof of Theorem \ref{Main Thm}.
\begin{rem}
	By definition $FP(s,\nu) \subseteq C\big([s,T],\mathcal{P}\big)$ for all $s \in [0,T]$ and $\nu \in \mathcal{P}$. Since $\mathcal{P}$ is Polish and $[s,T]$ is compact, elements in $C\big([s,T],\mathcal{P}\big)$ are even uniformly continuous and in particular every element in  $FP(s,\nu)$ is uniquely determined by its values in a countable, dense subset of $[s,T]$. More precisely, for $(\gamma_q)_{q \in \mathbb{Q}^T_s} \in \mathcal{P}^{\mathbb{Q}^T_s},$ there is at most one $(\bar{\gamma}_t)_{t \in [s,T]} \in FP(s,\nu)$ such that $\bar{\gamma}_q=\gamma_q$ for all $q \in \mathbb{Q}^T_s$. The advantage of this consideration is that $\mathcal{P}^{\mathbb{Q}^T_s}$ is, in contrast to $C\big([s,T],\mathcal{P}\big)$, metrizable and hence continuity of a map $f: \mathcal{P}^{\mathbb{Q}^T_s} \to \mathbb{R}$ is equivalent to sequential continuity.
\end{rem}
Before we turn to the proof of Theorem \ref{Main Thm}, we need to introduce a few additional important results about Fokker-Planck-Kolmogorov equations and their connection to the corresponding martingale problem. We start by recalling the definition of a solution to the martingale problem associated to the given coefficients $a_{ij}$ and $b_i$.
\begin{dfn}\label{Def MP}
	Let $L=L_{A,b}$ denote the differential operator associated to the given coefficients $a_{ij}$ and $b_i$ from above. A \textit{(continuous) solution to the martingale problem associated to} $L$ with initial condition $(s,\nu) \in [0,T]\times\mathcal{P}$ is a measure $P \in \mathcal{P}\big(C([s,T],\mathbb{R}^d)\big)$ such that
	\begin{enumerate}
		\item [(i)] $P\circ\pi_s^{-1} = \nu$
		\item[(ii)] For every $\phi \in C^{1,2}_b\big((s,T)\times\mathbb{R}^d\big)$, the process $[s,T]\ni t \mapsto \phi(t,\pi_t)-\int_{s}^{t}\partial_u\phi(u,\cdot)\circ \pi_u+L_u\phi(u,\pi_u)du$ is a real-valued $P$-martingale on $C\big([s,T],\mathbb{R}^d\big)$ w.r.t. to the natural filtration $(\mathcal{F}_t)_{t \in [s,T]}$ on $C\big([s,T],\mathbb{R}^d\big)$, i.e. $\mathcal{F}_t := \sigma(\pi_u|u \in [s,t])$ for all $t \in [s,T]$.
	\end{enumerate}
\end{dfn}Here and below $\pi_t: C\big([s,T],\mathbb{R}^d\big) \to \mathbb{R}^d$ denotes the canoncial projection at time $t \in [s,T]$.
\begin{prop}\label{MP to FPK}
	Let $P \in \mathcal{P}\big(C([s,T],\mathbb{R}^d)\big)$ be a solution to the martingale problem associated to $L$ with initial condition $(s,\nu)$. Then $(P\circ \pi_t^{-1})_{t \in [s,T]} \in FP(s,\nu)$.
\end{prop}
\begin{proof}
	Obviously $(P \circ \pi_t^{-1})_{t \in [s,T]}$ is a Borel curve of probability measures. Due to the continuity of the canoncial projections $\pi_t$ for $t \in [s,T]$, $(P\circ \pi_t^{-1})_{t \in [s,T]}$ is clearly narrowly continuous. Using the martingale property of Definition \ref{Def MP} (ii), we obtain by integration with $P$, Fubini's theorem, a change of variables for image measures and (i) of the previous definition:
	$$\int_{\mathbb{R}^d} \phi dP_t - \int_{s}^{t} \int_{\mathbb{R}^d} L_u\phi dP_u du = \int_{\mathbb{R}^d} \phi dP_s = \int_{\mathbb{R}^d}\phi d\nu$$ for every $\phi \in C^{\infty}_0(\mathbb{R}^d)$, where we abbreviated $P \circ \pi_t^{-1}$ by $P_t$.
\end{proof} The following theorem by Trevisan gives sort of an inverse of the above proposition. We point out that Trevisan's result (c.f. Theorem 2.5 in \cite{trevisan2016}), which is an extension of an earlier work by Figalli (c.f. \cite{FIGALLI2008109}), does not require any continuity or boundedness of the coefficients (instead of the latter, in \cite{trevisan2016} global integrability of $a_{ij}$ and $b_i$ against $d\gamma_t(x)dt$ over $[s,T]\times \mathbb{R}^d$ is required for any solution $\gamma$).
\begin{theorem}\label{Superposition}(Superposition principle by Trevisan/Figalli) Let $(\gamma_t)_{t \in [s,T]} \in FP(s,\nu)$. Then there exists a solution $P \in \mathcal{P}\big(C([s,T],\mathbb{R}^d)\big)$ to the martingale problem associated to $L$ with start in $(s,\nu)$ such that $P\circ \pi_t^{-1} = \gamma_t$ for all $t \in [s,T]$.
	
\end{theorem}

The following proposition is a minor extension of Theorem 1.4.6 in \cite{stroock2007multidimensional} and provides a convenient tool to check whether a given family $\mathcal{M}\subseteq\mathcal{P}\big(C([s,T],\mathbb{R}^d)\big)$ is precompact w.r.t. the topology of weak convergence of probability measures. Below, $B_l(x) \subseteq \mathbb{R}^d$ denotes the euclidean ball with radius $l \geq 0$ centered around $x \in \mathbb{R}^d$. For $f: \mathbb{R}^d\to \mathbb{R}$ and $a \in \mathbb{R}^d$, $f_a(x):= f(x-a)$ is the translate of $f$ by the vector $a$.
\begin{prop}\label{precompact}
	A family $\mathcal{M} \subseteq \mathcal{P}\big(C([s,T],\mathbb{R}^d)\big)$ is precompact if and only if both of the following hold:
	\begin{enumerate}
		\item [(i)] $\underset{l \to \infty}{lim}\underset{P \in \mathcal{M}}{sup}P\circ \pi_s^{-1}\big(B_l(0)^c\big) =0$
		\item[(ii)] For every non-negative $f \in C^{\infty}_0(\mathbb{R}^d)$, there exists a constant $c_f \geq 0$, which does not depend on $P \in \mathcal{M}$, such that $f_a(\pi_t)+c_ft$ is a non-negative submartingale w.r.t. to the natural filtration on $C\big([s,T],\mathbb{R}^d\big)$ for every $P \in \mathcal{M}$ and every $a \in \mathbb{R}^d$.
	\end{enumerate}
\end{prop}
\begin{proof}
	It suffices to note that the proof of Theorem 1.4.6. in \cite{stroock2007multidimensional} still holds when one replaces $C\big(\mathbb{R}_+,\mathbb{R}^d\big)$ by $C\big([s,T],\mathbb{R}^d\big)$ for arbitrary $0 \leq s \leq T < \infty$. 
\end{proof}
We now state a crucial compactness result for the set of solutions to the martingale problem associated to $L$ with initial condition $(s,\nu)$. Essentially, this result is formulated as part of Lemma 12.2.1 in \cite{stroock2007multidimensional}. However, as this lemma only covers the compactness for deterministic initial conditions - i.e. $\nu = \delta_x$ for $x \in \mathbb{R}^d$ - in the case of time-independent coefficients, for the convenience of the reader we decided to give a proof for the more general version, which we shall need below.
\begin{prop}\label{S.V.}
	Let $MP(s,\nu) \subseteq \mathcal{P}\big(C([s,T],\mathbb{R}^d)\big)$ be the set of all solutions to the martingale problem associated to $L$ with start in $(s,\nu)$. Then $MP(s,\nu)$ is a compact subset of $\mathcal{P}\big(C([s,T],\mathbb{R}^d)\big)$, endowed with the topology of weak convergence of measures.
\end{prop}
\begin{proof}
	Using Proposition \ref{precompact}, we first show that $MP(s,\nu)$ is precompact. Indeed, since $P\circ\pi_s^{-1} = \nu$ for all $P \in MP(s,\nu)$ and every Borel probability meaure on $\mathbb{R}^d$ is tight, we obtain $(i)$ of Proposition \ref{precompact}. Concerning $(ii)$, note that for non-negative $f \in C^{\infty}_0(\mathbb{R}^d)$, due to the boundedness of $a_{ij}$ and $b_i$ and since $f$ is compactly supported, there is a constant $c_f \geq 0$ such that 
	\begin{equation}\label{constant}
	L_uf(\pi_u) \geq -c_f
	\end{equation}for all $u \in [s,T]$. Hence, $[s,T] \ni t \mapsto \int_{s}^{t}L_uf(\pi_u)du +c_f  t$ is non-negative and increasing and thus for any $P\in MP(s,\nu)$ the process $f(\pi_t)+c_f t$ is a non-negative submartingale on $[s,T]$ w.r.t. the natural filtration on $C([s,T],\mathbb{R}^d)$ under $P$. It is clear that the same constant $c_f$ works for all translates $f_a$ as well, since the boundedness of all coefficients and $f \in C^{\infty}_0(\mathbb{R}^d)$ yields that (\ref{constant}) holds for every $f_a$. Hence Proposition \ref{precompact} applies and $MP(s,\nu) \subseteq \mathcal{P}\big(C([s,T],\mathbb{R}^d)\big)$ is precompact.\\
	\\It remains to show the closedness of $MP(s,\nu)$. Therefore, let $\{P_n\}_{n \in \mathbb{N}} \subseteq MP(s,\nu)$ such that $P_n \underset{n \to \infty}{\to}P$ weakly in $\mathcal{P}\big(C([s,T],\mathbb{R}^d)\big)$. First of all it is obvious that $P \circ \pi_s^{-1} = \nu$, since the canoncial projection $\pi_s: C([s,T],\mathbb{R}^d) \to \mathbb{R}^d$ is continuous and $P_n\circ \pi_s^{-1} = \nu$ for all $n \in \mathbb{N}$. In order to prove $P \in MP(s,\nu)$, we show
	\begin{equation*}
	\int\bigg[\phi(t,\pi_t)-\int_{s}^{t}\partial_u\phi(u,\cdot)\circ \pi_u+ L_u\phi(u,\pi_u)du \bigg]G_r dP = \int\bigg[\phi(r,\pi_r)-\int_{s}^{r}\partial_u\phi(u,\cdot)\circ\pi_u+L_u\phi(u,\pi_u)du \bigg]G_r dP
	\end{equation*}for all $0 \leq s \leq r \leq t \leq T$, $\phi \in C^{1,2}_b\big((s,T)\times\mathbb{R}^d\big)$ and every continuous, bounded $\mathcal{F}_r$-measurable $G_r: C([s,T],\mathbb{R}^d)\to \mathbb{R}$. But as $P_n \in MP(s,\nu)$, this holds for every $n \in \mathbb{N}$ and since $$\bigg(\phi(z,\pi_z)-\int_{s}^{z}\partial_u\phi(u,\cdot)\circ\pi_u+L_u\phi(u,\pi_u)du\bigg)G_r: C([s,T],\mathbb{R}^d) \to \mathbb{R}$$ is bounded and continuous for every $z\in [s,T]$ (the latter due to a classical criterion for continuity of parameter-dependent integrals, using the continuity of the canonical projections $\pi_u$ plus the continuity in $x \in \mathbb{R}^d$ and boundedness of $a_{ij}$, $b_i$ and the boundedness of $\phi$), the weak convergence $P_n \underset{n \to \infty}{\to}P$ implies the desired equality. Therefore $P \in MP(s,\nu)$ and the proof is complete.
\end{proof}We now turn to the proof of Theorem \ref{Main Thm}. \\
\\
\textit{Proof of Theorem \ref{Main Thm}:} Let $\{f_n\}_{n \in \mathbb{N}} \subseteq C_b(\mathbb{R}^d)$ be a measure-determining family and $\eta$ a fixed enumeration, as presented in Definition \ref{Def enumeration}. Below we adopt all notations of Definition \ref{Def enumeration}. Further fix $(s,\nu) \in [0,T]\times\mathcal{P}$. We define the following values, maps and sets, where for abbreviation we occasionally write $\mu$ instead of $(\mu_t)_{t \in [s,T]}$. The map $J_s$ is as in Defnition \ref{Def J}.
\begin{align*}
u_{m^s_0}(s,\nu) &:= \underset{\mu \in J_s(FP(s,\nu))}{sup}\int f_{n_{m^s_0}}d\mu_{q_{m^s_0}},\\ G^{s,\nu}_{m^s_0}: J_s\big(FP(s,\nu)\big) &\to \mathbb{R},\,\, (\mu_q)_{q \in \mathbb{Q}^T_s} \mapsto \int f_{n_{m^s_0}}d\mu_{q_{m^s_0}}, \\M_{m^s_0}(s,\nu)&:= G^{s,\nu}_{m^s_0}(s,\nu)^{-1}(\{u_{m^s_0}\})
\end{align*}and iteratively
\begin{align*}
u_{m^s_{k+1}}(s,\nu) &:= \underset{\mu \in M_{m^s_{k}}(s,\nu)}{sup}\int f_{n_{m^s_{k+1}}}d\mu_{q_{m^s_{k+1}}},\\ G^{s,\nu}_{m^s_{k+1}}: M_{m^s_k} &\to \mathbb{R},\,\, (\mu_q)_{q \in \mathbb{Q}^T_s} \mapsto \int f_{n_{m^s_{k+1}}}d\mu_{q_{m^s_{k+1}}}, \\M_{m^s_{k+1}}(s,\nu)&:= G^{s,\nu}_{m^s_{k+1}}(s,\nu)^{-1}(\{u_{m^s_{k+1}}\}).
\end{align*}We make the following observations: Since we assume $FP(s,\nu) \neq \emptyset$ for all $(s,\nu)\in [0,T]\times \mathcal{P}$, we have $J_s\big(FP(s,\nu)\big) \neq \emptyset.$ Therefore, and because each $f_n$ is bounded and for $\mu \in J_s\big(FP(s,\nu)\big)$ every marginal $\mu_q$ is a probability measure, we have $u_{m^s_0}(s,\nu) \in \mathbb{R}$. By the same argument we have $u_{m^s_{k+1}}(s,\nu) \in \mathbb{R}$ for $k \in \mathbb{N}_0$, provided $M_{m^s_{k}} \neq \emptyset$. Moreover, $G^{s,\nu}_{m^s_0}$ is continuous, because $(\mu_q)_{q \in \mathbb{Q}^T_s} \mapsto \mu_{q_{m^s_0}}$ is clearly continuous from $J_s\big(FP(s,\nu)\big)$ with the induced product topology of weak convergence to $\mathcal{P}$ with the topology of weak convergence and $\mu_{q_{m^s_0}} \mapsto \int f_{n_{m^s_0}}d\mu_{q_{m^s_0}}$ is continuous by definition of the topology of weak convergence of measures and since $f_n \in C_b(\mathbb{R}^d)$ for every $n \in \mathbb{N}$. If $M_{m^s_{k}} \neq \emptyset$, then the same holds true for $G^{s,\nu}_{m^s_{k+1}}$.\\
\\Our aim is to prove $\big|\underset{k\in \mathbb{N}_0}{\bigcap}M_{m^s_k}(s,\nu)\big|=1.$ Naturally we prove this in two steps. We start by proving $\big|\underset{k\in \mathbb{N}_0}{\bigcap}M_{m^s_k}(s,\nu)\big|\leq1$: Indeed, if $(\mu^1_q)_{q \in \mathbb{Q}^T_s}$ and $(\mu^2_q)_{q \in \mathbb{Q}^T_s}$ are two elements in $\underset{k\in \mathbb{N}_0}{\bigcap}M_{m^s_k}(s,\nu)$, then, since $\{m^s_k|k \in \mathbb{N}_0\}$ is an enumerating sequence of $\mathbb{N}\times \mathbb{Q}^T_s$, we have $\int f_n d\mu^1_q = \int f_n d\mu^2_q$ for all $(n,q)\in \mathbb{N}\times \mathbb{Q}^T_s$. Since $\{f_n\}_{n \in \mathbb{N}}$ is measure-determining, we obtain $\mu^1_q = \mu^2_q$ for all $ q \in  \mathbb{Q}^T_s$, which implies $\big|\underset{k\in \mathbb{N}_0}{\bigcap}M_{m^s_k}(s,\nu)\big|\leq1$.\\ Even more, by definition of $M_{m^s_0}(s,\nu)$, both $(\mu^1_q)_{q \in \mathbb{Q}^T_s}$ and $(\mu^2_q)_{q \in \mathbb{Q}^T_s}$ are elements of $J_s\big(FP(s,\nu)\big)$ and thus, there is a unique element $(\bar{\mu}_t)_{t \in [s,T]} \in FP(s,\nu)$ such that $\mu^1_q = \bar{\mu}_q = \mu^2_q$ for all $q \in \mathbb{Q}^T_s$. We conclude: There exists at most one element $\mu = (\mu_t)_{t \in [s,T]} \in FP(s,\nu)$ such that $(\mu_q)_{q \in  \mathbb{Q}^T_s} \in J_s\big(FP(s,\nu)\big)$.\\ \\
\\We now show $\big|\underset{k\in \mathbb{N}_0}{\bigcap}M_{m^s_k}(s,\nu)\big|\geq1:$We start by showing that $J_s\big(FP(s,\nu)\big) \subseteq \mathcal{P}^{\mathbb{Q}^T_s}$ is compact.  \\By Proposition \ref{S.V.}, $MP(s,\nu)$ is a compact subset of $\mathcal{P}\big(C([s,T],\mathbb{R}^d)\big)$ with the topology of weak convergence. Now define the following map:
$$
\Lambda: \mathcal{P}\big(C([s,T],\mathbb{R}^d)\big) \to C\big([s,T],\mathcal{P}\big),\,\,\,Q \mapsto (Q\circ \pi_t^{-1})_{t \in [s,T]}.
$$ We prove continuity of $\Lambda$ by letting $Q^n \underset{n \to \infty}{\to}Q$ in $\mathcal{P}\big(C([s,T],\mathbb{R}^d)\big)$, i.e. $\int GdQ^n \underset{n \to \infty}{\to}\int GdQ$ for all $G \in C_b\big(C([s,T],\mathbb{R}^d)\big)$. Clearly, for every $g \in C_b(\mathbb{R}^d)$ and $t \in [s,T]$, $g \circ \pi_t \in C_b\big(C([s,T],\mathbb{R}^d)\big)$ and hence in particular $\int g dQ^n_t \underset{n \to \infty}{\to}\int gdQ_t$ for all such $g$ and $t$. This means $(Q^n_t)_{t \in [s,T]} \underset{n \to \infty}{\to} (Q_t)_{t\in [s,T]}$ in $C\big([s,T],\mathcal{P}\big)$ (endowed with the product topology of weak convergence of measures) and therefore $\Lambda$ is continuous. Hence $\Lambda\big(MP(s,\nu)\big) \subseteq C\big([s,T],\mathcal{P}\big)$ is compact. By Proposition \ref{MP to FPK} (giving "$\subseteq$") and Theorem \ref{Superposition} (giving $"\supseteq"$), we clearly have \begin{equation}\label{important_1}
\Lambda\big(MP(s,\nu)\big) = FP(s,\nu)
\end{equation}and therefore $FP(s,\nu) \subseteq C\big([s,T],\mathcal{P}\big)$ is compact.\\
\\By Remark \ref{Cont J} and the above, the map $$J_s \circ \Lambda: \mathcal{P}\big(C([s,T],\mathbb{R}^d)\big) \to \mathcal{P}^{\mathbb{Q}^T_s},\,\,\, Q \mapsto (Q\circ \pi_q^{-1})_{q \in \mathbb{Q}^T_s}$$is continuous as well. Therefore and by (\ref{important_1}), $J_s\big(FP(s,\nu)\big) \subseteq \mathcal{P}^{\mathbb{Q}}_{s,T}$ is compact. Thus, by the continuity of $G^{s,\nu}_{m^s_0}$,  $M_{m^s_0}(s,\nu) \subseteq \mathcal{P}^{\mathbb{Q}^T_s}$ is non-empty and compact.
\\ Repeating the same arguments with $J_s\big(FP(s,\nu)\big)$ replaced by $M_{m^s_0(s,\nu)}$, $G^{s,\nu}_{m^s_0}$ by $G^{s,\nu}_{m^s_1}$ and $u_{m^s_0}(s,\nu)$ by $u_{m^s_1}(s,\nu)$, we obtain that $M_{m^s_1}(s,\nu) \subseteq M_{m^s_0}(s,\nu)$ is non-empty and compact as well.\\ Iterating this procedure, we obtain a sequence of non-empty, compact sets $\big(M_{m^s_k}(s,\nu)\big)_{k \in \mathbb{N}_0}$ with $M_{m^s_{k+1}}(s,\nu) \subseteq M_{m^s_{k}}(s,\nu) \subseteq \mathcal{P}^{\mathbb{Q}^T_s}$ for all $k \in \mathbb{N}$. Finally, this implies $\big|\underset{k\in \mathbb{N}_0}{\bigcap}M_{m^s_k}(s,\nu)\big|\geq1.$\\
\\Let us recapitulate what we have achieved so far: For each initial condition $(s,\nu) \in [0,T]\times \mathcal{P}$, we have characterized a unique element $\mu^{s,\nu} = (\mu^{s,\nu}_t)_{t \in [s,T]} \in FP(s,\nu)$ in the sense that it is the unique element in $FP(s,\nu)$ such that  $(\mu^{s,\nu})_{q \in \mathbb{Q}^T_s} \in \underset{k\in \mathbb{N}_0}{\bigcap}M_{m^s_k}(s,\nu).$ We consider $\mu^{s,\nu}$ as an extremal element in $FP(s,\nu)$, because by construction it is "iteratively maximal" in the sense that $(\mu^{s,\nu}_q)_{q \in \mathbb{Q}^T_s} \in M_{m^s_k}$ for every $k \in \mathbb{N}_0$.\\ \\To conclude the proof, it remains to show that the family $(\mu^{s,\nu})_{(s,\nu)\in [0,T]\times \mathcal{P}}$ has the desired flow property, i.e. it fulfills (\ref{flow-prop}) of Definition \ref{Def flow prop}. In order to prove this, let us fix $0\leq s \leq r \leq T$ and $\nu \in \mathcal{P}$ and let $\mu^{s,\nu}=(\mu^{s,\nu}_t)_{t \in [s,T]} \in FP(s,\nu)$ be the unique iteratively maximal solution for the initial condition $(s,\nu)$ as described in the previous passage. Consider the initial condition $(r,\mu^{s,\nu}_r)$ and let $\gamma = (\gamma_t)_{t \in [r,T]}\in FP(r,\mu^{s,\nu}_r)$ be the unique iteratively maximal solution for the initial condition $(r,\mu^{s,\nu}_r)$, i.e. in our notation $\gamma_t = \mu_t^{r,\mu^{s,\nu}_r}$. We need to show
\begin{equation}\label{Eq for flow}
\gamma_t=\mu^{s,\nu}_t \text{ for all }t\in [r,T].
\end{equation}
We proceed as follows: Define $\zeta = (\zeta_t)_{t \in [s,T]}$ by
$$\zeta_t := \begin{cases} 
\mu^{s,\nu}_t, & t \in [s,r] \\
\gamma_t, & t \in [r,T]
\end{cases},$$which is a well-defined Borel curve of probability measures, since $\gamma_r=\mu^{s,\nu}_r$ and since both $\mu^{s,\nu}$ and $\gamma$ are Borel curves. Clearly $\zeta \in FP(s,\nu)$: Indeed, it is obvious that $ [s,T] \ni t \mapsto \zeta_t$ is narrowly continuous, $\zeta_s=\nu$ holds and for $f \in C^{\infty}_0(\mathbb{R}^d)$ and $t \in [s,T]$ we have
\begin{align*}
\int f d\zeta_t -\int fd\nu &= \int fd\zeta_t-\int fd\mu^{s,\nu}_r + \int fd\mu^{s,\nu}_r - \int fd\nu \\&=  \int_{r}^{t}\int_{\mathbb{R}^d}L_uf(x)d\gamma_u(x)du + \int_{s}^{r}\int_{\mathbb{R}^d}L_uf(x)d\mu^{s,\nu}_u(x)du \\&= \int_{s}^{t}\int_{\mathbb{R}^d}L_uf(x)d\zeta_udu, \text{ if }t\in [r,T] 
\end{align*}and
\begin{align*}
\int f d\zeta_t -\int fd\nu = \int_{s}^{t}\int_{\mathbb{R}^d}L_uf(x)d\mu^{s,\nu}_udu = \int_{s}^{t}\int_{\mathbb{R}^d}L_uf(x)d\zeta_udu, \text{ if }t \in [s,r[,
\end{align*}which gives $\zeta \in FP(s,\nu)$ by Definition \ref{Def sol FPK}. Therefore, by virtue of the characterizing property of $\mu^{s,\nu}$ among all elements of $FP(s,\nu)$, we have
\begin{equation}\label{important_4}
\int f_{n_{m^s_0}}d\mu^{s,\nu}_{q_{m^s_0}} \geq \int f_{n_{m^s_0}}d\zeta_{q_{m^s_0}}. 
\end{equation}If $q_{m^s_0} \in [s,r]$, then $\zeta_{q_{m^s_0}}=\mu^{s,\nu}_{q_{m^s_0}}$ and we have equality in (\ref{important_4}). If $q_{m^s_0} \in \,]r,T]$, then $q_{m^s_0}=q_{m^r_0}$ and by the characterizing property of $\gamma$ among all elements of $FP(r,\mu^{s,\nu}_r)$ and since $(\mu^{s,\nu}_t)_{t \in [r,T]} \in FP(r,\mu^{s,\nu}_r)$, we have
$$\int f_{n_{m^s_0}}d\mu^{s,\nu}_{q_{m^s_0}} \leq \int f_{n_{m^s_0}}d\gamma_{q_{m^s_0}} =\int f_{n_{m^s_0}}d\zeta_{q_{m^s_0}}$$ and hence we have equality in (\ref{important_4}) in any case.  Now consider $m^s_1$: Since we have equality in (\ref{important_4}), both $(\mu^{s,\nu}_q)_{q \in \mathbb{Q}^T_s}$ and $(\zeta_q)_{q \in \mathbb{Q}^T_s}$ belong to $M_{m^s_0}(s,\nu)$. Hence, using the characterization of $\mu^{s,\nu}$ again, we obtain
\begin{equation}\label{important_6}
\int f_{n_{m^s_1}}d\mu^{s,\nu}_{q_{m^s_1}} \geq \int f_{n_{m^s_1}}d\zeta_{q_{m^s_1}},
\end{equation}clearly with equality if $q_{m^s_1} \in [s,r]$. If $q_{m^s_1} \in \, ]r,T]$ and $q_{m^s_0} \in [s,r]$, then $m^s_1 = m^r_0$ and we must have
\begin{equation}\label{important_7}
\int f_{n_{m^s_1}}d\mu^{s,\nu}_{q_{m^s_1}} \leq \int f_{n_{m^s_1}}d\gamma_{q_{m^s_1}} = \int f_{n_{m^s_1}}d\zeta_{q_{m^s_1}}
\end{equation}by the characterizing property of $\gamma$ and hence equality in (\ref{important_6}). If $q_{m^s_0}, q_{m^s_1} \in\, ]r,T]$, then $m_0^s=m_0^r$, $m^s_1=m^r_1$ and both $(\mu^{s,\nu}_q)_{q \in \mathbb{Q}^T_r}$ and $(\gamma_q)_{q \in \mathbb{Q}^T_r}$ are in $M_{m^r_0}(r,\mu^{s,\nu}_r)$ and we also obtain (\ref{important_7}). Hence, equality in (\ref{important_6}) holds in any case. Iterating this procedure yields
\begin{equation*}
\int f_{n_{m^s_k}}d\mu^{s,\nu}_{q_{m^s_k}} = \int f_{n_{m^s_k}}d\zeta_{q_{m^s_k}}
\end{equation*}for all $k \in \mathbb{N}_0$. As $(m^s_k)_{k \in \mathbb{N}_0}$ is the enumerating sequence of $\mathbb{N}\times \mathbb{Q}^T_s$ and $\{f_n\}_{n \in \mathbb{N}}$ is measure-determining on $\mathbb{R}^d$, this yields
$$\mu^{s,\nu}_q = \zeta_q \text{ for all }q\in \mathbb{Q}^T_s, $$so in particular $\mu^{s,\nu}_q = \gamma_q$ for all $q \in \mathbb{Q}^T_r$.
Since both $(\gamma_q)_{q \in \mathbb{Q}^T_r}$ and $(\mu^{s,\nu}_q)_{q \in \mathbb{Q}^T_r}$ belong to $J_r\big(FP(r,\mu^{s,\nu}_r)\big)$, we obtain (\ref{Eq for flow}). \qed\\
\\ \begin{rem}
	We point out that if we perform the procedure of the above proof for a different measure-determining family $\{g_n\}_{n \in \mathbb{N}} \subseteq C_b(\mathbb{R}^d)$ and/or with a different enumeration $\delta$ instead of $\eta$, we may obtain a different family of solutions with the flow property. This also becomes apparent in the next theorem and its proof.\\ Above that, in principle, one could also consider a different dense, countable subset of $[0,T]$ instead of $\mathbb{Q}^T_0$. This could also lead to a differnt solution family with the flow property. 
\end{rem}The following theorem is an interesting consequence of the method we used to construct a flow of solutions within the proof of Theorem \ref{Main Thm}. 
\begin{theorem}\label{second main thm}
	Let all assumptions of Theorem \ref{Main Thm} be in force. Then the following are equivalent:
	\begin{enumerate}
		\item [(i)] The FPK-eq. is well-posed among narrowly continuous probability solutions, i.e. $\big|FP(s,\nu)\big| =1$ for all $ (s,\nu) \in [0,T]\times \mathcal{P}$.
		\item[(ii)] There exists exactly one family of solutions $(\mu^{s,\nu})_{(s,\nu) \in [0,T]\times \mathcal{P}}$ with the flow-property of Definition \ref{Def flow prop}.
	\end{enumerate}
\end{theorem}
\begin{proof} The implication $(i) \implies (ii)$ follows immediately, because the existence of a flow follows by Theorem \ref{Main Thm} and due to well-posedness, there can obviously not be two differing flows.\\
	\\Consider $(ii) \implies (i)$.
	Assume, under the general assumption $\big|FP(s,\nu)\big| \geq 1$ for all $(s,\nu) \in [0,T]\times \mathcal{P}$, that the FPK-eq. is not well-posed, i.e. there exists an initial condition $(\bar{s},\bar{\nu})$ such that $\big| FP(\bar{s},\bar{\nu})\big| \geq 2$. Let $\{f_n\}_{n \in \mathbb{N}} \subseteq C_b(\mathbb{R}^d)$ be a measure-determining family to which with each $f_n$ also $-f_n$ belongs, $\eta:\mathbb{N}\times \mathbb{Q}^T_0 \to \mathbb{N}_0$ an enumeration and let $u_{m^s_k}(s,\nu)$ and $M_{m^s_k}(s,\nu)$ be defined as in the proof of Theorem \ref{Main Thm} for all $k \in \mathbb{N}_0$, $s \in [0,T]$ and $\nu \in \mathcal{P}$. Let $(\mu^{s,\nu})_{(s,\nu) \in [0,T]\times \mathcal{P}}$ be the selected flow subject to this data as constructed in the proof of Theorem \ref{Main Thm}.\\
	By assumption, for $(\bar{s},\bar{\nu})$ there is $\gamma \in FP(\bar{s},\bar{\nu})$ such that $\gamma_{\bar{t}} \neq \mu^{\bar{s},\bar{\nu}}_{\bar{t}}$ for some $\bar{t} \in [s,T]$. By narrow continuity of elements in $FP(s,\nu)$, we may w.l.o.g. assume $\bar{t} \in \mathbb{Q}^T_s$. Hence, there must be a member of the measure-determining sequence, say $f_{\bar{n}}$, such that
	$$\int f_{\bar{n}}d\mu^{\bar{s},\bar{\nu}}_{\bar{t}} \neq \int f_{\bar{n}}d\gamma_{\bar{t}}.$$Let us assume w.l.o.g. 
	\begin{equation}\label{one}
	\int f_{\bar{n}}d\gamma_{\bar{t}} > \int f_{\bar{n}}d\mu^{\bar{s},\bar{\nu}}_{\bar{t}},
	\end{equation} else we consider $-f_{\bar{n}}$ instead, which by assumption also belongs to the family $\{f_n\}_{n \in \mathbb{N}}$. \\
	Now consider the same measure-determining sequence $\{f_n\}_{n \in \mathbb{N}}$, but a different enumeration $\delta$ such that according to this enumeration $(f_{n_0},q_{n_0}) = (f_{\bar{n}},\bar{t})$ and denote the corresponding flow of solutions constructed as in the proof of Theorem \ref{Main Thm} by $(\beta^{s,\nu})_{(s,\nu) \in [0,T]\times \mathcal{P}}$. We will show that this flow is not the same as $(\mu^{s,\nu})_{(s,\nu) \in [0,T]\times \mathcal{P}}$. Indeed, we have
	$$\int f_{\bar{n}}d\beta^{\bar{s},\bar{\nu}}_{\bar{t}} \geq \int f_{\bar{n}}d\gamma_{\bar{t}} > \int f_{\bar{n}}d\mu^{\bar{s},\bar{\nu}}_{\bar{t}}.$$Here the first inequality holds by the characterizing property of $\beta^{\bar{s},\bar{\nu}}$ and the choice of the enumeration $\delta$ and the second one is just (\ref{one}). Hence, we cannot have
	$$\int f_{\bar{n}}d\mu^{\bar{s},\bar{\nu}}_{\bar{t}} = \int f_{\bar{n}}d\beta^{\bar{s},\bar{\nu}}_{\bar{t}}$$ and thereby $\mu^{\bar{s},\bar{\nu}}_{\bar{t}} \neq \beta^{\bar{s},\bar{\nu}}_{\bar{t}}$, which shows that the two flows $(\mu^{s,\nu})_{(s,\nu) \in [0,T]\times \mathcal{P}}$ and $(\beta^{s,\nu})_{(s,\nu) \in [0,T]\times \mathcal{P}}$ are not identical. This finishes the proof.
\end{proof}
\begin{rem}
	We would like to point out the following observation: From the proof of Theorem \ref{Main Thm} it is clear that instead of $FP(s,\nu)$, we can also consider arbitrary closed, non-empty subsets $C(s,\nu) \subseteq FP(s,\nu)$ with the following property: Whenever $(\gamma_t)_{t \in [s,T]} \in C(s,\nu)$, then $(\gamma_t)_{t \in [r,T]} \in C(r,\gamma_r)$ for $0 \leq s \leq r \leq T$. Performing the same method as in the proof of Theorem \ref{Main Thm}, we then construct a family of solutions $(\mu^{s,\nu})_{(s,\nu) \in [0,T]\times \mathcal{P}}$ with the flow property such that $\mu^{s,\nu} \in C(s,\nu)$ for every $(s,\nu)$. This could provide a useful tool to impose a priori additional properties on the members of the flow family. It is then obvious that also Theorem \ref{second main thm} holds when each $FP(s,\nu)$ is replaced by such $ C(s,\nu)$.
\end{rem}


\section*{Acknowledgements}
I am deeply indebted to my supervisor Prof. Michael Röckner, who pointed out this interesting topic to me and with whom I shared various fruitful discussions concering the proof of the first main theorem. Further, financial support by the German Science Foundation DFG (IRTG 2235) is gratefully acknowledged.

\bibliography{FPK_flow_paper}

\end{document}